\documentclass[review]{elsarticle}
\usepackage{lineno,hyperref}
\usepackage{amsmath,amssymb,amsthm}
\usepackage{graphicx}
\newtheorem{definition}{Definition}[section]
\newtheorem{theorem}{Theorem}[section]
\newtheorem{lemma}{Lemma}[section]
\newtheorem{proposition}{Proposition}[section]

\newtheorem{remark}{Remark}[section]
\numberwithin{equation}{section}
\journal{arXiv}

\begin{document}
	
	\begin{frontmatter}
		
		\title{The exponential turnpike property for periodic linear quadratic optimal control problems in infinite dimension
			\tnoteref{mytitlenote}}
		\tnotetext[mytitlenote]{This work was partially supported  by the National Natural Science Foundation of China under the grant 11971363, and by the Fundamental Research Funds for the Central Universities under the grant 2042023kf0193.}

		\author[mysecondaryaddress]{Emmanuel Tr\'elat}
		\ead{emmanuel.trelat@sorbonne-universite.fr}
		
		\medskip
		
		\author[my3address]{Xingwu Zeng}
		\ead{xingwuzeng@whu.edu.cn}
	
	\author[my4address]{Can Zhang}
		\ead{zhangcansx@163.com}

\address[mysecondaryaddress]{Sorbonne Universit\'e, CNRS, Universit\'e Paris Cit\'e, Inria, Laboratoire Jacques-Louis Lions (LJLL), F-75005 Paris, France}
\address[my3address]{School of Mathematics and Statistics, Wuhan University, Wuhan 430072, China}
\address[my4address]{School of Mathematics and Statistics, Wuhan University, Wuhan 430072, China}

\begin{abstract}
	In this paper, we establish an exponential periodic turnpike property for linear quadratic optimal control problems governed by periodic systems in infinite dimension.  We show that the optimal trajectory converges exponentially to a periodic orbit when the time horizon tends to infinity.  Similar results are obtained for the optimal control and adjoint state. Our proof is based on the large time behavior of solutions  of operator differential Riccati equations with periodic coefficients.
\end{abstract}


\end{frontmatter}


\tableofcontents


\section{Introduction}

The turnpike property was first observed in the context of finite-dimensional discrete-time optimal growth problems by economists (see, e.g., \cite{DGSW, LM}).
This property reflects the fact that, for an optimal control problem for which the time horizon is large enough, its optimal solution stays most of the time close to a turnpike set.

In many cases, the turnpike set is a singleton consisting of the minimizer of an optimal steady-state problem (see, e.g., \cite{GG1,  GG2, GTZ, GL, PZ, T1, TZZ, TZ, za}). But turnpike sets may be more complicated and may consist for instance of periodic orbits (see, e.g., \cite{FFOW, T1, TZ1}). In \cite{PAS}, Samuelson established a periodic turnpike property for a finite-dimensional optimal growth problem in economics, where the minimization functional is periodic in time. In \cite{TZ1,ZMG}, the authors considered the periodic turnpike property in the context of dissipativity. In \cite{SWY,SY}, the authors established a turnpike property for finite-dimensional stochastic LQ optimal control problems.

The simplest time-varying linear systems are those for which the coefficients are time periodic. Periodic linear systems arise frequently as the result of linearizing a nonlinear system along a periodic orbit. In \cite{Xu}, the author established a characterization of periodic stabilization in terms of a detectability inequality for  a linear periodic control system in a Hilbert space with a bounded control operator. The problem of tracking periodic signals for finite and infinite-dimensional linear periodic systems has been considered in \cite{AL} and \cite{AI}.

The authors of \cite{TZZ} established a periodic turnpike property for linear quadratic (LQ) optimal control problems with periodic tracking terms for time-invariant systems under exponential stabilizability and detectability assumptions as well as some smallness assumptions.
The main ingredient in \cite{TZZ} is a dichotomy transformation for the solutions of the algebraic Riccati and Lyapunov equations.

In the present paper we investigate the periodic exponential turnpike property for infinite-dimensional LQ  optimal control problems with periodic coefficients, under appropriate periodic exponential stabilizability and detectability assumptions. 
The goal is to show that, except at the extremities of the time interval, the optimal trajectory (also, control and adjoint state) remains exponentially close to a periodic optimal trajectory, which itself is characterized as the optimal solution of an associated periodic optimal control problem.
This widely generalizes the above-mentioned result of \cite{TZZ} and the technique of proof is entirely different. 

Our approach here exploits the exponential stabilizability of the evolution operator resulting from the operator differential Riccati equation, and an exponential estimate between the solution of the differential Riccati equation with a zero terminal value  
to its periodic one, when the time horizon tends to infinity (see Proposition~\ref{dudic}). This new exponential estimate is at the heart of the proof. The techniques that we use are inspired from earlier works by Da Prato and Ichikawa (see \cite{D, PI, PI1, AI}).

The paper is organized as follows. 
In Section~\ref{free}, we introduce the LQ optimal control problem for periodic systems, and we state our main result, Theorem~\ref{periodictarget}.
In Section~\ref{auxiresults}, we first establish some instrumental properties of the evolution operator resulting from the Riccati equation, and then we state and prove Proposition~\ref{dudic}. Theorem~\ref{periodictarget} is proved in Section~\ref{proof}. Section~\ref{comment} gives some further comments.


\section{Main result}\label{free}


\subsection{Formulation of the LQ optimal control problem}

Throughout the paper, we use $U$ and $H$ to denote Hilbert spaces, and use $\left\langle\cdot,\cdot \right\rangle$ and $\|\cdot\|$ to denote the inner product and norm in all spaces without causing any confusion.
We denote by $L(U, H)$ the space of linear bounded operators from $U$ to $H$.
We set $\Sigma(H)=\left\{P \in L(H) ; P=P^*\right\}$ and $\Sigma^{+}(H)=\{P \in \Sigma(H) ; P \geq 0\}$, where $ P^* $ denotes the adjoint operator of $ P $.
We denote by
$C\big([a, b] ; L(H)\big)$  the space of all mappings $P:[a,b] \rightarrow L(H)$ such that 
$P(\cdot) x$ is continuous for any $x \in H$ with the norm 
$
\|P\|=\sup \big\{\|P(t)\|: t \in[a, b]\big\}.
$

Given any $T>0$, we consider the time-periodic linear quadratic optimal control problem
\begin{equation*}\label{12181}
(LQ)^T:\; \inf_{u\in L^2(0,T;U)}\;\;\frac{1}{2}\int_0^T 
\Big(\big\|C(t)\big( y(t)-y_d(t)\big) \big\|^2+\big \|Q^{1/2}(t)u(t)\big\|^2\Big)\,dt,
\end{equation*}
where $y_d(\cdot) \in C([0, +\infty);H)$ is $\theta$-periodic, i.e. $y_d(t+\theta)=y_d(t) $ for every $ t \in \mathbb{R}$, and $(y(\cdot),u(\cdot)) \in C([0,T];H)\times L^2(0,T;U)$ satisfies 
the controlled system
\begin{equation}\label{2.1}
\left\{
\begin{split}
	&\dot{y}(t)=A(t)y(t)+B(t)u(t),\;\;t \in [0,T],\\
	&y(0)=y_0\in H,
\end{split}\right.
\end{equation}
where $ A(\cdot) $, $ B(\cdot) $, $ C(\cdot) $ and $ Q(\cdot) $ are $\theta$-periodic, i.e. $A(t+\theta)=A(t)$ for every $t \in \mathbb{R}$, and the same for the others.
Here, for each $t$,
$A(t)$ is a linear unbounded operator in $H$, $ B(\cdot) \in C\big(\mathbb{R}, L(U, H)\big) $, $ C(\cdot) \in C\big(\mathbb{R}, \Sigma^{+}(H)\big) $, $ Q(\cdot) \in C\big(\mathbb{R}, \Sigma^{+}(U)\big) $, and there exists $\varepsilon>0$ such that $Q(t) \geq \varepsilon I$ for every $ t \in \mathbb{R}$. 

Moreover, we assume that $ A(\cdot) $ generates an evolution system, i.e., there exists a strongly continuous mapping $  U_A :\{(t, s) \mid t \geqslant s\} \rightarrow L(H) $ such that $$ \frac{\partial}{\partial t} U_A(t, s)   =A(t) U_A(t, s) , \;\; U_A(s, s)  =I,\;\; U_A(t, r)U_A(r, s)=U_A(t, s)\;\; \text{for all}\;\; 0 \leqslant s \leqslant r \leqslant t, $$
and there exists Yosida's approximation $ A_n(t)=n^2(n-A(t))^{-1}-n I $ for every $t \geqslant 0 $, for $ n $ large enough so that $$ \lim _{n \rightarrow \infty} U_{A_n}(t, s) x=U_A(t, s) x,\; \forall x \in H, \; \text{uniformly on}  \; \{(t, s) \mid 0 \leqslant s \leqslant t \leqslant \theta\}. $$

\begin{remark}
The conditions about the operators are fulfilled under the usual hypotheses of Tanabe and Kato-Tanabe (see, for instance, \cite{AT}, \cite{CP}, \cite{IK}, \cite{L}, \cite{AP} and \cite{HT}). Sometimes, the strongly continuous mapping $U_A(\cdot, \cdot)$ is called the evolution operator associated with $A(\cdot)$. 
It can be noted that Yosida's approximation is satisfied if, for instance, the family of $ A(\cdot) $ is dissipative or quasi-dissipative (see, for instance, \cite{IK}), which is a standard assumption in that context.
The existence of $U_{A_n}(\cdot, \cdot)$, the evolution operator associated with $A_n (\cdot)$, is clear, since $A_n(t)$ is bounded for each $t \geqslant 0$ (see, for instance, \cite{PI1}).
\end{remark}

\begin{remark}
The system $\eqref{2.1}$ has a unique mild solution given by
\begin{equation*}
	y(t)=U_A(t,0)y_0+\int_0^t U_A(t,s)B(s)u(s)ds.
\end{equation*}
Moreover, the evolution operator $U_A(\cdot,\cdot)$ is $\theta$-periodic, i.e.,
\begin{equation*}
	U_A(t+\theta,s+\theta)=U_A(t,s),\;\;\forall \;0\leqslant s\leqslant t.
\end{equation*}
\end{remark}

The problem $(LQ)^T$ has a unique optimal solution denoted by $(y^T(\cdot),u^T(\cdot))$.\footnote{The optimal solution depends on $T$, and we add a superscript to emphasize this fact.} Moreover, according to \cite[Chapter 4, Theorem 1.6]{LY}, there exists an adjoint state $\lambda^T(\cdot)\in C([0,T];H)$ such that
\begin{equation}
\left\{
\begin{split}
	\dot{y}^T(t) &= A(t)y^T(t)-B(t)Q^{-1}(t)B^*(t)\lambda^T(t), \\
	\dot{\lambda}^T(t) &= -C^*(t)C(t)y^T(t)-A^*(t)\lambda^T(t)+C^*(t)C(t)y_d(t),
\end{split}
\right.
\end{equation}
in the mild sense along $[0,T]$, with the two-point boundary condition
\begin{equation}
y^T(0)=y_0\;\;\;\text{and}\;\;\;
\lambda^T(T)=0.
\end{equation}
Furthermore, the optimal control is given by 
\begin{equation*}
u^T(t)=-Q^{-1}(t)B^*(t)\lambda^ T (t),\;\;\text{for a.e.}\;\;t\in[0,T].
\end{equation*}

\subsection{The periodic turnpike theorem}
In order to define the turnpike set, 
we consider the periodic optimal control problem:
\begin{equation*}\label{xiao-2}
(Per)_\theta:\; \inf\;\;  
\frac{1}{2}\int_0^\theta \left( \big\|C(t)\big( y(t)-y_d(t)\big) \big\|^2+\big \|Q^{1/2}(t)u(t)\big\|^2\right) dt,
\end{equation*}
over all pairs  $(y(\cdot),u(\cdot))\in C([0,\theta];H)\times L^2(0,\theta;U)$ satisfying
\begin{equation}\label{2.6}
\left\{
\begin{split}
	&\dot{y}(t)=A(t)y(t)+B(t)u(t),\;\;t\in[0,\theta],\\
	&y(0)=y(\theta).\\
\end{split}\right.
\end{equation}

The system $\eqref{2.6}$ does not necessarily have a periodic solution for a given $\theta$-periodic control $u \in L^2(0,\theta; U)$. This happens when no Floquet exponent of $A(\cdot)$ is equal to
$1$, i.e., when $1$ belongs to the resolvent set\footnote{The resolvent set $\rho(A)$ of a closed linear operator $A$ is the set of all complex numbers $\lambda$ such that the operator $A_\lambda=\lambda I-A$ has a bounded inverse.} $\rho(U_A(\theta, 0))$ of $U_A(\theta, 0)$, or that  $U_A(\cdot, \cdot)$ is exponentially stable (see, for instance, \cite[Proposition 2.1]{PI}). In fact, if $1$ belongs to the resolvent set  $\rho(U_A(\theta, 0))$, then the system $\eqref{2.6}$ has a unique mild solution for every $\theta$-periodic control $u \in L^2(0,\theta; U)$, given by
\begin{equation*}
y(t)=U_A(t,0)\left[I-U_A(\theta,0)\right]^{-1} \int_0^\theta U_A(\theta,s)B(s)u(s) ds+\int_0^t U_A(t,s)B(s)u(s)ds.
\end{equation*}
Actually, under some specific conditions, the exponentially stability of $U_A(\cdot,\cdot)$ implies that $1$ belongs to the resolvent set $\rho(U_A(\theta, 0))$ of $U_A(\theta, 0)$ (see, for instance, \cite[Corollary 2.1]{BP}).

Existence and uniqueness of a periodic solution for such periodic problems (Per)$_\theta$ under certain conditions, as well as  necessary and sufficient conditions for optimality, have been  widely studied in the existing literature (see, for instance, \cite{BP}, \cite{D} or \cite[Chapter 4, Proposition 5.2]{LY}).
It is well known that if $\left(y_\theta(\cdot),u_\theta(\cdot)\right)$ is an optimal pair for $(\text{Per})_\theta$, then there exists an adjoint variable $\lambda_\theta(\cdot)\in C([0,\theta];H)$ such that
\begin{equation}\label{xiaoextremalsystLQ}
\left\{
\begin{split}
	\dot{y}_\theta(t) &= A(t)y_\theta(t)-B(t)Q^{-1}(t)B^*(t)\lambda_\theta(t), \\
	\dot{\lambda}_\theta(t) &= -C^*(t)C(t)y_\theta(t)-A^*(t)\lambda_\theta(t)+C^*(t)C(t)y_d(t),
\end{split}
\right.
\end{equation}
in the mild sense along $[0,\theta]$, with the periodic boundary conditions
\begin{equation}\label{du12061}
y_\theta(0)=y_\theta(\theta)\;\;\;\text{and}\;\;\;
\lambda_\theta(0)=\lambda_\theta(\theta).
\end{equation}
Moreover, the optimal periodic control is given by 
\begin{equation}\label{xiao61718}
u_\theta(t)=-Q^{-1}(t)B^*(t)\lambda_\theta(t),\;\;\text{for a.e.}\;\;t\in[0,\theta].
\end{equation}
We say that $\big(y_\theta(\cdot),u_\theta(\cdot),\lambda_\theta(\cdot)\big)$ is the periodic solution of (Per)$_\theta$.

In our main result (Theorem~\ref{periodictarget}) hereafter, the assumptions that we do indeed ensure the existence and uniqueness of the periodic solution of (Per)$_\theta$,
which is the turnpike set around which the exponential turnpike property is established. To introduce these assumptions, we next recall the definitions
of exponential stabilizability and detectability in the time-periodic framework.
\begin{definition}\label{du1206}
The periodic pair $(A(\cdot),B(\cdot))$ is called exponentially $\theta$-periodic stabilizable 
if there exists a feedback $\theta$-periodic function
$K_1(\cdot)\in C\big(\mathbb R;L(H,U)\big)$ such that the following closed-loop system
is exponentially stable:
\begin{equation*}
	\dot y(t)=\big(A(t)+B(t)K_1(t)\big)y(t),\;\;t>0.
\end{equation*} 
The periodic pair $(A(\cdot), C(\cdot))$ is called exponentially $\theta$-periodic detectable 
if $(A^*(\cdot),C^*(\cdot))$ is exponentially $\theta$-periodic stabilizable, i.e., there exists a feedback $\theta$-periodic function
$K_2 (\cdot)\in C\big(\mathbb R;L(H)\big)$ such that the following closed-loop system
is exponentially stable:
\begin{equation*}
	\dot y(t)=\big(A(t)+K_2(t)C(t)\big)y(t),\;\;t>0.
\end{equation*} 
\end{definition}

\begin{theorem}\label{periodictarget}
Assume that $(A(\cdot),B(\cdot))$ is exponentially $\theta$-periodic stabilizable,  and that $(A(\cdot),C(\cdot))$ is exponentially $\theta$-periodic detectable.
Then the problem $ (Per)_\theta $ has a unique solution
$\big(y_\theta(\cdot),u_\theta(\cdot),\lambda_\theta(\cdot)\big)$ (extended by $\theta$-periodicity over the whole real line). Moreover, we
have the exponential periodic turnpike property: there exist two positive constants $C$ and $\nu$ such that for any $T > 0$ large enough, the optimal solution $(y^T(\cdot),u^T(\cdot), \lambda^T(\cdot))$ of 
$(LQ)^T$ satisfies 
\begin{equation}\label{10271}
	\left\Vert y^T(t)-y_\theta(t)\right\Vert+\left\Vert u^T(t)-u_\theta(t)\right\Vert+\left\Vert \lambda^T(t)-\lambda_\theta(t)\right\Vert
	\leqslant C\Big( e^{-\nu t}+e^{-\nu (T-t)}\Big) ,
\end{equation}
for every $t\in[0,T]$. 
\end{theorem}

\begin{remark}
The exponential decay constant $\nu$ in $\eqref{10271}$ is the exponential stability rate for the evolution operator resulting from the operator Riccati equation in $\eqref{riccat2}$ below, and the constant $C$ in $\eqref{10271}$ is of the form $C_1\big(\|y_\theta(0)-y_0\|+\left( 1+\|y_0\| \right) \big)$, where the constant $C_1$ does not depend on $y_0$ and $y_d(\cdot)$.
\end{remark}

\subsection{Examples}
\subsubsection{Periodic heat equation}\label{example}
Let $\Omega$ be an open and bounded domain of $\mathbb{R}^n$ with a $ C^2 $ boundary $\partial\Omega$. Let $\omega \subseteq \Omega$ be a non-empty open subset with its characteristic function $\chi_\omega$. Given  $T>0$ and $y_0 \in  L^2(\Omega)$,  we consider the following optimal control problem:
\begin{equation}\label{5.1}
\inf_{u\in L^2(0,T;L^2(\Omega))} \frac{1}{2} \int_0^T  \int_{\Omega}\left[|y(x, t)-y_d(x, t)|^2+|u(x, t)|^2\right] d x d t,
\end{equation}
subject to
\begin{equation*}
\left\{\begin{array}{l}
	\partial_t y(x, t)=\Delta y(x, t)-a(x, t) y(x, t)+\chi_\omega(x)u(x, t)\quad \text { in } \Omega \times(0, T) ,\\
	y(x, t)=0 \quad \text { on } \partial\Omega \times(0, T) , \\ 
	y(x, 0)=y_0 \quad \text { in } \Omega,
\end{array}\right.
\end{equation*}
where $a \in C(\bar{\Omega} \times[0, \infty))$ and $y_d \in C(\bar{\Omega} \times[0, \infty))$ are $\theta$-periodic with respect to the time variable, and $\Delta$ is the classical  Laplace operator. The periodic optimal control problem (Per)$_\theta$ reads as
\begin{equation}\label{5.2}
\inf_{u\in L^2(0,\theta;L^2(\Omega))} \frac{1}{2} \int_0^\theta  \int_{\Omega}\left[|y(x, t)-y_d(x, t)|^2+|u(x, t)|^2\right] d x d t,
\end{equation}
subject to
\begin{equation*}
\left\{\begin{array}{l}
	\partial_t y(x, t)=\Delta y(x, t)-a(x, t) y(x, t)+\chi_\omega(x)u(x, t)\quad \text { in } \Omega \times(0, T) ,\\
	y(x, t)=0 \quad \text { on } \partial\Omega \times(0, T) , \\ 
	y(x, 0)=y(x, \theta) \quad \text { in } \Omega.
\end{array}\right.
\end{equation*}

We take $H$ = $U$ = $L^2(\Omega)$, $C(t)$ = $Q(t)$ = $I$, $B(t)$ = $\chi_\omega$ for each $t \in [0, T]$, and 
\begin{equation*}
A(t)z=\Delta z-a(\cdot, t) z, \;\; \forall z \in H^2(\Omega) \cap H_0^1(\Omega). 
\end{equation*}
Clearly, $D(A(t))=H^2(\Omega) \cap H_0^1(\Omega)$, and $A(\cdot)$ generates an evolution operator in $L^2(\Omega)$.
By \cite[Corollary 2.1]{WX}, $(A(\cdot),B(\cdot))$ is exponentially $\theta$-periodic stabilizable, and $(A(\cdot),C(\cdot))$ is exponentially $\theta$-periodic detectable. Therefore, according to Theorem~\ref{periodictarget}, the periodic optimal control problem (Per)$_\theta$ has a unique solution, and the optimal control problem under consideration has the periodic exponential turnpike property \eqref{10271}.

\subsubsection{Periodic wave equation}\label{example1}
Given $ \ell >0 $,
let $ y_d \in C([0,\infty );L^2(0,\ell )) $ be a 1-periodic tracking trajectory, i.e., $ y_d (\cdot,t) = y_d (\cdot,t+1) $ for any $ t \geqslant 0 $. Given  $T>1$,  we consider the following optimal control problem:
\begin{equation}\label{5.2}
	\inf_{u\in L^2(0,T;L^2(0,\ell))} \frac{1}{2} \int_0^T  \int_0^\ell |y_t (x, t)-y_d(x, t)|^2  dxdt+\frac{1}{2} \int_0^T  \int_0^\ell |u(x, t)|^2 d x d t,
\end{equation}
over all possible $ (y,u) \in \left(C([0,T];H_0^1 (0,\ell ))\cap C^1 ([0,T];L^2 (0,\ell ))\right) \times L^2(0,T;L^2(0,\ell)) $ satisfying
\begin{equation*}
	\left\{\begin{array}{l}
		\partial_{tt} y(x, t)=\partial_{xx} y(x, t)-a(x, t) \partial_t y(x, t)+u(x, t),\quad (x,t) \in (0,\ell) \times(0, T) ,\\
		y(0, t)=y(\ell,t) = 0, \quad t\in  (0, T) , \\ 
		y(x, 0)=y_0 (x),\;\;y_t (x,0) = y_1 (x), \quad x\in (0,\ell),
	\end{array}\right.
\end{equation*}
where $a \in C([0,\ell] \times[0, \infty))$ is $ 1 $-periodic with respect to the time variable, $ y_0 \in H_0^1 (0,\ell ) $ and $ y_1 \in L^2 (0,\ell) $. 

We take $ H=H_0^1 (0,\ell) \times L^2 (0,\ell ),\; U=L^2 (0,\ell) $. For each $ t \in[0,T] $, define
\begin{equation*}
	A(t) = \begin{pmatrix}
		0  &I\\ \partial_{xx}  &0
		\end{pmatrix}+\begin{pmatrix}
		0  &0\\0  &-a(\cdot,t)
	\end{pmatrix} := A_w + R(t),
\end{equation*}
\begin{equation*}
	B(t)=\begin{pmatrix}
		0\\ I
	\end{pmatrix},	C(t)=\begin{pmatrix}
	0 \quad 0 \\
	0  \quad I
\end{pmatrix},
\end{equation*}
where $ I $ is the identity operator, and $ D(A(t))=\left(H^2 (0,\ell) \cap H_0^1 (0,\ell)\right)\times H_0^1(0,\ell) $. There exist $ K_s \in L(H,U) $ and $ K_d \in L(H) $ such that $A_w +B(\cdot)K_s$ and $ A_w ^* +C^* (\cdot)K_d $ is exponentially stable (see, e.g., \cite[Example 4.3]{PI}). Moreover note that $ A^* (\cdot) =-A_w +R(\cdot) $, $ B(\cdot) B^* (\cdot)= C(\cdot) $, and $ C^* (\cdot)C(\cdot)  = C(\cdot) $. Hence $(A(\cdot),B(\cdot))$ is exponentially $1$-periodic stabilizable with $ K_1 (\cdot):=K_s -B^* (\cdot)R(\cdot) $ and $(A(\cdot), C(\cdot))$ is exponentially $1$-periodic detectable with $ K_2 (\cdot):=K_d - C(\cdot)R(\cdot) $.  Therefore, according to Theorem~\ref{periodictarget}, the optimal control problem under consideration has the periodic exponential turnpike property \eqref{10271}.

\subsection{A numerical simulation}
In this section, we provide a simple example to numerically illustrate the periodic turnpike phenomenon in the finite-dimensional case.
Given any $ T > 0 $, we consider the LQ optimal control problem of
minimizing the cost functional
\begin{equation*}
	\frac{1}{2}\int_0^T 
	\Big[\big(4-\sin^2t-\cos t\big)\big( y(t)-\cos t\big)^2+u^2(t)\Big]\;dt
\end{equation*}
for the one-dimensional control system
\begin{equation*}
	\dot y (t) = \sin t \;y(t) + u(t),\quad t\in (0,T),
\end{equation*}
with a fixed initial condition 
$y(0)=0.1$. 

To fit in our framework, we set for each $t\in(0,T)$
\begin{equation*}
	A(t) = \sin t,\quad B(t)=1,\quad Q(t) = 1,\quad C^* (t) C(t) = 4-\sin^2 t-\cos t,\quad y_d (t) = \cos t.
\end{equation*}
Using MATLAB, 
\begin{itemize}
	\item First, we compute the periodic solution $ P_\theta $ of the periodic Riccati equation \eqref{riccat2}.
	\item Second, we compute the periodic solution $ r_\theta $ of the equation \eqref{r_theta}.
	\item Third, we compute the periodic turnpike $ (y_\theta,\lambda_\theta,u_\theta) $ by  
	\eqref{poc}-\eqref{adp}.
\end{itemize}
The optimal extremal $ (y^T , \lambda^T, u^T) $, resulting from the first-order optimality system derived from the Pontryagin maximum principle, can be computed in time $ T = 50 $, by the MATLAB function \texttt{bvp4c}.
The turnpike property can be observed in the Figure~\ref{fig:f1} below. As expected, except for the transient
initial and final arcs, the extremal $ (y^T , \lambda^T, u^T) $ (in red) remains close to
the periodic turnpike $ (y_\theta,\lambda_\theta,u_\theta) $ (in blue).
\begin{figure}
	\centering
	\includegraphics[scale=0.25]{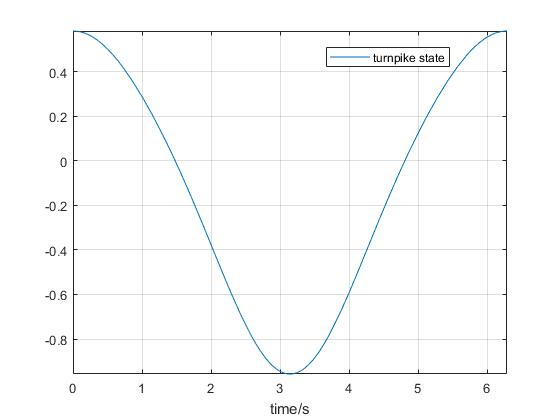}
	\includegraphics[scale=0.25]{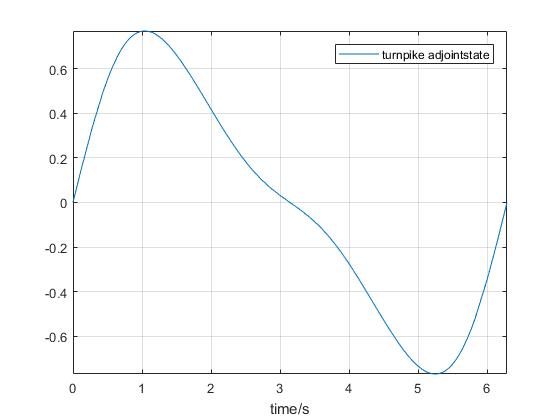}
	\includegraphics[scale=0.25]{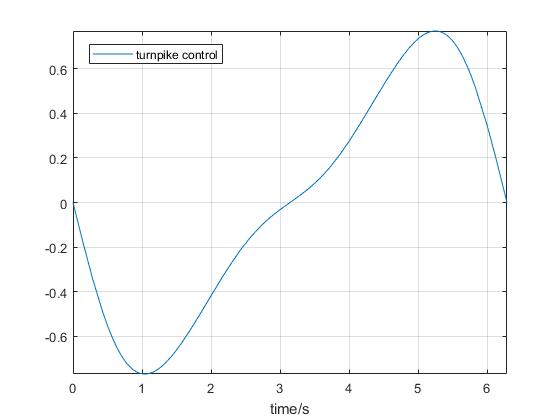}
	\includegraphics[scale=0.25]{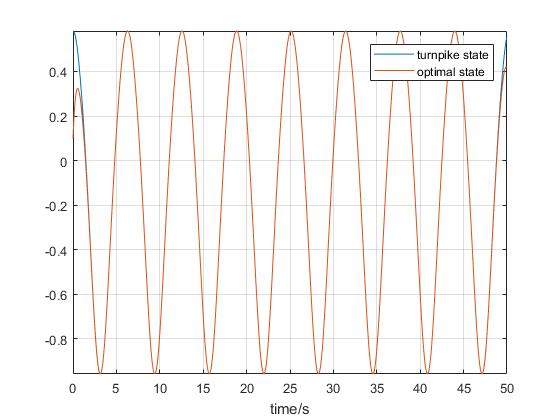}
	\includegraphics[scale=0.25]{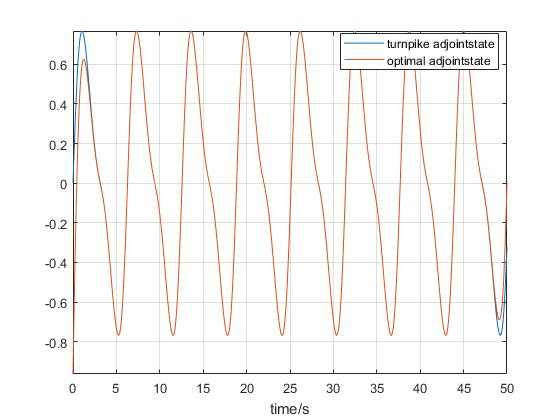}
	\includegraphics[scale=0.25]{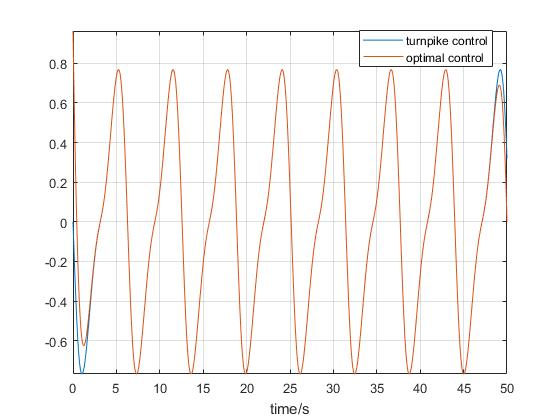}
	\caption{Example of a periodic turnpike. \label{fig:f1}}	
\end{figure}

\section{Auxiliary results}\label{auxiresults}

\subsection{Reminders on differential Riccati equations}
We first state two preliminary results, see their proofs, for instance, in \cite[Proposition 3.1 and Theorems 3.8]{D}, \cite[Proposition 3.4]{PI1}, respectively. More precisely, Lemma \ref{ric} and Lemma \ref{ric1} are about the existence and uniqueness of solutions to the Riccati  differential equation. Next, we recall the corresponding value function and the solution of the  problem $(LQ)^T$ and the problem (Per)$_\theta$ given by  the extended Riccati equation.
To do this, we first introduce the differential Riccati equation
\begin{equation}\label{riccat1}\left\{
	\begin{split}
		&\dot{P}(t)+A^*(t)P(t)+P(t)A(t)-P(t)B(t)Q^{-1}(t)B^*(t)P(t)+C^*(t)C(t)=0,\;\;t\in (0, T), \\
		&P(T)=0 ,
	\end{split}
	\right.
\end{equation}
and the differential periodic Riccati equation
\begin{equation}\label{riccat2}
	\left\{
	\begin{split}
		&\dot{P}(t)+A^*(t)P(t)+P(t)A(t)-P(t)B(t)Q^{-1}(t)B^*(t)P(t)+C^*(t)C(t)=0,\;\;t\in (0, \theta), \\
		&P(\theta)=P(0).
	\end{split}
	\right.
\end{equation}

\begin{lemma}\label{ric}
Equation $\eqref{riccat1}$ admits a unique mild solution\footnote{We say that $P(\cdot)\in C([0,T];H)$ is a mild solution of the final value problem $\eqref{riccat1}$ if for each $t \in [0,T]$ and each $h \in H$,\\ $P(t)h=U_A^*(T,t)P(T)U_A(T,t)h-\int_t^TU_A^* (s,t)\left( P(s)B(s)Q^{-1}(s)B^*(s)P(s)-C^*(s)C(s)\right)U_A(s,t)h \,ds$.} $P^{T}(\cdot) \in C\big([0,T];\Sigma^{+}(H)\big)$.
\end{lemma}

\begin{lemma}\label{ric1}
Assume that $(A(\cdot),B(\cdot))$ is exponentially $\theta$-periodic stabilizable, and that $(A(\cdot),C(\cdot))$ is exponentially $\theta$-periodic detectable. Then Equation $\eqref{riccat2}$ admits a unique $\theta$-periodic mild solution $P_\theta(\cdot) \in C\big([0,\theta];\Sigma^{+}(H)\big)$. Moreover, for each $ t \in \mathbb{R}$ and $ x\in H $, $$ \lim_{T \to +\infty} P^T (t)x=P_\theta (t)x. $$
\end{lemma}

We next define the following function
\begin{equation*}
v^T(t,x)=\frac{1}{2}\langle P^T(t) x, x\rangle+\langle r^T(t), x\rangle+s^T(t),\;\;(t, x) \in[0, T] \times H,
\end{equation*}
where $P^T(\cdot) \in C\big([0,T];\Sigma^{+}(H)\big)$ is the mild solution of the differential Riccati equation \eqref{riccat1}
and $r^T(\cdot) \in C([0,T];H)$ is the solution of
\begin{equation*}
\left\{
\begin{split}
	&\dot{r}(t)=-\left( A(t)-B(t)Q^{-1}(t)B^*(t)P^T(t)\right)^* r(t)+C^*(t)C(t)y_d(t),\\
	&r(T)=0,
\end{split}
\right.
\end{equation*}
and for each $t \in [0,T]$, $s^T(t)$ is given by
\begin{equation*}
s^T(t)=-\frac{1}{2}\int_t^T\left[\left\|Q^{-1/2}(s) B^*(s) r^T(s)\right\|^2-\left\|C(s)y_d(s)\right\|^2\right] ds.
\end{equation*}

Then, one can easily check that $v^T$ is the value function of the problem $(LQ)^T$ (see \cite[Part III]{TE} for more details).
Furthermore, the optimal control of the problem $(LQ)^T$ is given by (see, e.g., \cite[Chapter 6, Theorem 5.5]{LY}) 
\begin{equation}\label{Toc}
u^T(t)=-Q^{-1}(t)B^*(t)\left( P^T(t)y^T(t)+r^T(t)\right) , \;\;\text{ for a.e.}\;\;t\in[0,T],
\end{equation}
the optimal trajectory $y^T(\cdot) \in C\big([0,T],H\big)$ is the solution of
\begin{equation}\label{y_T_equ}
\dot{y}^T(t) = \big(A(t)-B(t)Q^{-1}(t)B^*(t)P^T(t)\big)y^T(t)-B(t)Q^{-1}(t)B^*(t)r^T(t),
\end{equation}
with the initial condition
\begin{equation*}
y^T(0)=y_0 ,
\end{equation*}
and the optimal adjoint state $\lambda^T(\cdot)\in C([0,T];H)$ is given by
\begin{equation}\label{adT}
\lambda^T(t)=P^T(t)y^T(t)+r^T(t),\;\;\text{ for a.e.}\;\;t\in[0,T].
\end{equation}

Similarly, assuming  that $(A(\cdot),B(\cdot))$ is exponentially $\theta$-periodic stabilizable, and that $(A(\cdot),C(\cdot))$ is exponentially $\theta$-periodic detectable, the value function $v^\theta:[0,\theta]\times H \rightarrow \mathbb R$ corresponding to the problem (Per)$_\theta$ is given by (see, e.g., \cite[Proposition 4.1]{D})
\begin{equation*}
v^\theta(t,x)=\frac{1}{2}\langle P_\theta(t) x, x\rangle+\langle r_\theta(t), x\rangle+s_\theta(t),\;\;(t, x) \in[0, \theta] \times H, 
\end{equation*}the optimal control is given by (see, e.g., \cite[Theorem 4.3]{D} or \cite[Chapter 6, Theorem 5.5]{LY}) 
\begin{equation}\label{poc}
u_\theta(t)=-Q^{-1}(t)B^*(t)\left( P_\theta(t)y_\theta(t)+r_\theta(t)\right) , \;\;\text{ for a.e.}\;\;t\in[0,\theta],
\end{equation}
the optimal trajectory $y_\theta(\cdot)\in C([0,\theta];H)$ is the unique mild solution of
\begin{equation}\label{2.12}
\dot{y}_\theta(t) = \big(A(t)-B(t)Q^{-1}(t)B^*(t)P_\theta(t)\big)y_\theta(t)-B(t)Q^{-1}(t)B^*(t)r_\theta(t),
\end{equation}
with a periodic condition
\begin{equation*}
y_\theta(0)= y_\theta(\theta),
\end{equation*}
and the optimal adjoint state $\lambda_\theta(\cdot)\in C([0,\theta];H)$ is given by 
\begin{equation}\label{adp}
\lambda_\theta(t)=P_\theta(t)y_\theta(t)+r_\theta(t),\;\;\text{ for a.e.}\;\;t\in[0,\theta].
\end{equation}
Here $P_\theta(\cdot) \in C\big([0,\theta];\Sigma^{+}(H)\big)$ is the unique mild solution of the differential periodic Riccati equation \eqref{riccat2}
and $r_\theta(\cdot) \in C([0,\theta];H)$ is the unique periodic solution of
\begin{equation}\label{r_theta}
\left\{
\begin{split}
	&\dot{r}(t)=-\left( A(t)-B(t)Q^{-1}(t)B^*(t)P_\theta(t)\right)^* r(t)+C^*(t)C(t)y_d(t),\;\;t\in (0, \theta),\\
	&r(\theta)=r(0),
\end{split}
\right.
\end{equation}
and $s_\theta(\cdot)$ is given by
\begin{equation*}
s_\theta(t)=-\frac{1}{2}\int_t^\theta\left[\left\|Q^{-1/2}(s) B^*(s) r_\theta(s)\right\|^2-\left\|C(s)y_d(s)\right\|^2\right] ds,\;\;t\in [0, \theta].
\end{equation*}

\subsection{Properties of the evolution operator}
 We next give two lemmas on the evolution operator. Lemma \ref{expsta} is on the exponential stability of the periodic evolution system, see the proof, for instance, in \cite[Theorem 1]{DR}; while Lemma \ref{varphi} gives the representation of the solution to the adjoint equation, and the proof is standard by a duality argument, hence we omit it. We use them to get the exponential stabilizability of the evolution operator.
\begin{lemma}\label{expsta}
	Equation $\eqref{2.1}$ with the null control is exponentially stable if and only
	if for each $h \in H$, there exists a finite constant $C(h)$, depending only on $h$, such that for all $t_0\geqslant0$, 
	\begin{equation}
		\int_{t_0}^{+\infty} \|U_A(t,t_0)h\|^2dt \leqslant C(h).
	\end{equation}
\end{lemma}

\begin{lemma}\label{varphi}
	If $\varphi(\cdot)$ is a mild solution (in backward time) for the adjoint equation
	\begin{equation}\label{varphieq}
		\left\{
		\begin{split}
			&\dot{\varphi}(t)=-A^*(t) \varphi(t) + g(t),\;\;t\in (0, T),\\
			&\varphi(T)=\varphi_T \in H,
		\end{split}
		\right.
	\end{equation}
	where $ g \in L^2 (0,T;H) $, then it can be expressed as 
	\begin{equation*}
		\varphi(t)=U^*_A(T,t)\varphi_T -\int_{t}^{T}U^*_A(\tau,t)g(\tau)d\tau , \;\; \forall t \in[0,T].
	\end{equation*}
\end{lemma}

\medskip

\subsection{Exponential convergence estimate}
\begin{lemma}\label{du120505}
Under the assumptions of  Lemma \ref{ric1}, the evolution operator $U_\theta(\cdot,\cdot)$ generated by $F_\theta(\cdot)=A(\cdot)-B(\cdot)Q^{-1}(\cdot)B^*(\cdot)P_\theta(\cdot)$ with $P_\theta(\cdot)$ being the $\theta$-periodic solution of $\eqref{riccat2}$ is exponentially stable, i.e., there exist two positive constants $M, \rho$ such that
\begin{equation}\label{prop1}
	\|U_\theta(t,s)\| \leqslant Me^{-\rho (t-s)} ,\; \; \forall \; 0\leqslant s \leqslant t.
\end{equation}
Moreover, the unique $\theta$-periodic solution $r_\theta(\cdot)\in C\big([0,\theta];H\big)$ of $\eqref{r_theta}$ is given by 
\begin{equation*}
	r_\theta(t)= -\int_{t}^{+\infty}U^*_\theta(\tau,t)C^*(\tau)C(\tau)y_d(\tau) d\tau,
\end{equation*}
and the optimal periodic trajectory of the problem (Per)$_\theta$ is given by
\begin{equation*}
	y_\theta(t)=-U_\theta(t,0)\int_{-\infty}^0 U_\theta(0,s) B(s)Q^{-1}(s)B^*(s)r_\theta(s)  ds-\int_0^t U_\theta(t,s) B(s)Q^{-1}(s)B^*(s)r_\theta(s)ds.
\end{equation*}
\end{lemma}
\begin{proof}
By \cite[Lemma 3.5]{D}, there exists a positive constant $c_1$ so that 
\begin{equation}\label{du12051}
	\int_{s}^{+\infty} \|U_\theta(t,s)h\|^2dt \leqslant c_1\|h\|^2,\;\; \forall s \geqslant 0 \;\;and\;\; \forall h \in H.
\end{equation}
The exponential stability property \eqref{prop1} of $U_\theta(\cdot,\cdot)$ then follows from Lemma \ref{expsta}.

We next claim that, for each $T_0\geqslant0$, any solution $z(\cdot) \in C\big([0,T_0];H\big)$ of \begin{equation}\label{adj}
	\left\{\begin{array}{l}
		\dot{z}(t)=-\left(A(t)-B(t)Q^{-1}(t)B^*(t)P_\theta(t)\right)^* z(t), \;\; 0\leqslant t\leqslant T_0, \\
		z(T_0) \in H,
	\end{array}\right.
\end{equation}
such that
\begin{equation}\label{adjexp}
	\|z(t)\| \leqslant Me^{-\rho(T_0-t)}\|z(T_0)\|, \;\;t \in [0,T_0].
\end{equation}
Indeed, by Lemma \ref{varphi}, the solution of $\eqref{adj}$ is
\begin{equation}
	z(s)=U^*_\theta(T_0,s)z(T_0),\;\; \forall s \in[0,T_0].
\end{equation} 
This and the  exponential stability of $U_\theta(\cdot,\cdot)$ imply $\eqref{adjexp}$. According to \cite[Proposition 1]{AI}, the solution of $\eqref{r_theta}$ is
\begin{equation*}
	r_\theta(t)= -\int_{t}^{+\infty}U^*_\theta(\tau,t)C^*(\tau)C(\tau)y_d(\tau) d\tau.
\end{equation*}
Therefore, by the exponential stability of $U_\theta(\cdot,\cdot)$ and the periodicity of $B(\cdot)$, $Q(\cdot)$, $C(\cdot)$ and $r_\theta(\cdot)$, and according to \cite[Proposition 2.1]{PI}, the unique periodic solution of $\eqref{2.12}$ is
\begin{equation*}
	\begin{aligned}
		y_\theta(t)&=U_\theta(t,0)y_\theta (0)-\int_0^t U_\theta(t,s) B(s)Q^{-1}(s)B^*(s)r_\theta(s)ds\\&=-U_\theta(t,0)\int_{-\infty}^0 U_\theta(0,s) B(s)Q^{-1}(s)B^*(s)r_\theta(s)  ds-\int_0^t U_\theta(t,s) B(s)Q^{-1}(s)B^*(s)r_\theta(s)ds.	
	\end{aligned}
\end{equation*}
\end{proof}

\medskip

We next establish an exponential estimate between  $P^T(\cdot)$ and $P_\theta(\cdot)$ when $ T $ is large enough. 
\begin{proposition}\label{dudic}
Under the assumptions of  Lemma \ref{ric1}, there exist two positive constants $M, \mu$ such that
\begin{equation}\label{prop2}
	\|P_\theta(t)-P^T(t)\| \leqslant Me^{-\mu (T-t)}, \; \; 0 \leqslant t \leqslant T,
\end{equation}
where $P_\theta(\cdot)$ is the $\theta$-periodic solution of $\eqref{riccat2}$, and $P^T(\cdot)$ is the solution of $\eqref{riccat1}$.
\end{proposition}
\begin{proof} 
	The argument is inspired by the proof of \cite[Proposition 3.2]{PI1}. Setting $R(\cdot)=P^T(\cdot)-P_\theta(\cdot)$ and  $F_\theta(\cdot)=A(\cdot)-B(\cdot) Q^{-1}(\cdot) B^*(\cdot) P_\theta(\cdot)$, we have
	\begin{equation*}
		\left\{\begin{array}{l}
			\dot{R}(t)+F^*_\theta(t) R(t)+R(t) F_\theta(t)-R(t) B(t) Q^{-1}(t) B^*(t) R(t)=0,\;\; 0\leqslant t \leqslant T,\\
			R(T)=P^T(T)-P_\theta(T)=-P_\theta(T)=-P_\theta(T-[T/\theta] \theta),
		\end{array}\right.
	\end{equation*}
	where $[x]$ is the integer part of $x$.
	
	For $ n $ large enough, let $ R_n (\cdot) $ be the solution of the final value problem
	\begin{equation*}
		\left\{\begin{array}{l}
			\dot{R}_n (t)+F^*_{\theta,n}(t) R_n (t)+R_n (t) F_{\theta,n}(t)-R_n (t) B(t) Q^{-1}(t) B^*(t) R_n (t)=0,\;\; 0\leqslant t \leqslant T,\\
			R_n (T)=P^T(T)-P_\theta(T)=-P_\theta(T)=-P_\theta(T-[T/\theta] \theta),
		\end{array}\right.
	\end{equation*}
	where $ F_{\theta,n}(\cdot)= A_n (\cdot)-B(\cdot) Q^{-1}(\cdot) B^*(\cdot) P_{\theta,n}(\cdot) $, $ A_n(\cdot) $ is the Yosida approximation of $ A(\cdot) $, and $ P_{\theta,n}(\cdot) $ is the solution of 
	\begin{equation*}
		\left\{
		\begin{split}
			&\dot{P}_{\theta,n}(t)+F^*_{\theta,n}(t)P_{\theta,n}(t)+P_{\theta,n}(t)F_{\theta,n}(t)+P_{\theta,n}(t)B(t)Q^{-1}(t)B^*(t)P_{\theta,n}(t)+C^*(t)C(t)=0,\;\;t\in (0, \theta), \\
			&P_{\theta,n}(\theta)=P_\theta (0).
		\end{split}
		\right.
	\end{equation*}
	For each $h \in H$, $\tau \in [0,T]$ and $ n $ large enough, let $y_n (\cdot) \in C\big([\tau,T],H\big)$ be the solution of 
	\begin{equation*}
		\left\{\begin{array}{l}
			\dot{y}_n (t)=F_{\theta,n}(t) y_n (t) , \;\; 0 \leqslant \tau \leqslant t \leqslant T,\\
			y_n (\tau)=h \in H.
		\end{array}\right.
	\end{equation*}
	A straightforward computation shows that
	\begin{equation*}
		\begin{aligned}
			\frac{d}{d t}\langle y_n (t), R_n (t) y_n (t)\rangle&=2\langle F_{\theta,n}(t) y_n (t), R_n (t) y_n (t)\rangle\\&\quad-\left\langle y_n (t),\left(F^*_{\theta,n}(t) R_n (t)+R_n (t) F_{\theta,n}(t)-R_n (t) B(t) Q^{-1}(t) B^*(t) R_n (t)\right)y_n (t)\right\rangle \\
			& =\left\langle y_n (t), R_n (t) B(t) Q^{-1}(t) B^*(t) R_n (t) y_n (t)\right\rangle \\
			& =\left\|Q^{-1/ 2}(t) B^*(t) R_n (t) y_n (t)\right\|^2 \geqslant 0,\;\; 0 \leqslant \tau \leqslant t \leqslant T.
		\end{aligned}
	\end{equation*}
	Integrating the above equation from $\tau$ to $T$ and letting $n$ go to infinity, we obtain that
	\begin{equation}\label{dh1}
		\langle y(T), R(T) y(T)\rangle \geqslant\langle y(\tau), R(\tau) y(\tau)\rangle.
	\end{equation}
	Denoting by $y(t):=U_\theta(t, \tau) h$, and by using the exponentially stability of 
	$U_\theta(\cdot,\cdot)$, we obtain from \eqref{dh1} that there exists a constant $ C>0 $ such that
	\begin{equation*}
		\begin{aligned}
			|\langle h, R(\tau) h\rangle| &\leqslant C|\left\langle U_\theta(T, \tau) h, R(T) U_\theta(T, \tau) h\right\rangle| \\
			& \leqslant C\|R(T)\|\left\|U_\theta(T, \tau) h\right\|^2 \\
			& \leqslant C\max _{0 \leq t \leq \theta}\left\|P_\theta(t)\right\|M^2 e^{-2 \rho(T-\tau)} \|h\|^2 .
		\end{aligned}
	\end{equation*}
	This implies that
	\begin{equation*}
		\|R(\tau)\| \leqslant M e^{-\mu(T-\tau)},\;\; 0\leqslant \tau \leqslant T,
	\end{equation*}
	for some positive constants $M$ and $\mu$, and it completes the proof.
\end{proof}

\begin{remark}
As it can be seen from the proof the exponent $\mu$ can be characterized as twice as much as the exponential stability rate for the evolution operator resulting from the Riccati equation in $\eqref{riccat2}$. 
The estimate is inspired from \cite[Part V, Proposition 4.3]{BPDM}, which is concerned about the exponential convergence of the solutions to the differential Riccati equations to its algebraic counterpart.  
\end{remark}

\begin{remark}
The inequality \eqref{dh1},  which is intrumental in the proof of Proposition~\ref{dudic}, is closely related to the dissipativity property, introduced 
in \cite{JCW}, and recently used to derive the turnpike property  (see, e.g., \cite{DGSW, GG1, GG2, T1, TZ1, ZMG}). Using the concept of dissipativity introduced in \cite[Definition 3.3]{ZMG} or \cite[Definition 3]{TZ1}, 
we prove in this remark that, under the assumptions of Theorem \ref{periodictarget}, the optimal control problem $ (LQ)^T $ is dissipative with respect to the $ \theta $-periodic optimal solution $ (y_\theta(\cdot),u_\theta(\cdot)) $ 
of  $(Per)_\theta$, with the \emph{supply rate} function
\begin{equation*}
	\begin{split}
		\omega(t,y,u) = \ell(t,y,u) - \ell(t,y_\theta(t),u_\theta(t)),\quad \forall (t,y,u)\in\mathbb{R}\times H \times U,
	\end{split}
\end{equation*}
where $ \ell(t,y,u) := \frac{1}{2}\Big(\big\|C(t)\big( y-y_d(t)\big) \big\|^2+\big \|Q^{1/2}(t)u\big\|^2\Big)$, and there exists a \emph{storage} function $ S:\mathbb{R}\times H \rightarrow \mathbb{R} $, $ \theta $-periodic in time, which is given by
\begin{equation*}
	S(t,y)=-\langle y-y_\theta(t),P_\theta(t)y_\theta(t)+r_\theta(t)\rangle,\;\;\;\forall (t,y)\in\mathbb R\times H,
\end{equation*} such that 
\begin{equation}\label{dissinequa}
	S(t_0,y(t_0))+\int_{t_0}^{t_1}\omega(t,y(t),u(t))\,dt \geqslant S(t_1,y(t_1)),\quad \text{for all } 0\leqslant t_0\leqslant t_1,
\end{equation}
and for all $ (y(\cdot),u(\cdot)) $ satisfying \eqref{2.1}.

To prove this fact, we first note that $ u_\theta(\cdot)=-Q^{-1}(\cdot)B^*(\cdot)\left( P_\theta(\cdot)y_\theta(\cdot)+r_\theta(\cdot)\right) $, where $ P_\theta(\cdot)$ is the mild solution of \eqref{riccat2} and $r_\theta(\cdot)$ is the periodic solution of \eqref{r_theta}. Following the approximation argument used in the proof of Proposition~\ref{dudic}, we get that
\begin{equation}\label{1}
	\begin{split}
		\frac{d}{dt}\langle P_\theta(t)y_\theta(t)+r_\theta(t),y_\theta(t)\rangle=&-2\ell(t,y_\theta(t),u_\theta(t))- \langle C(t)\big(y_\theta(t)-y_d(t)\big),C(t)
		y_d(t)\rangle, \; \; 0 \leqslant t \leqslant T.
	\end{split}
\end{equation}
Similarly, for any $ (y(\cdot),u(\cdot)) $ satisfying \eqref{2.1}, a straightforward  calculation gives	
\begin{equation*}
	\begin{split}
		\frac{d}{dt}\langle P_\theta(t)y_\theta(t)+r_\theta(t),y(t)\rangle=&
		-\langle C(t)\big(y_\theta(t)-y_d(t)\big),C(t)\big(y(t)-y_d(t)\big)\rangle-\langle Q(t)u_\theta(t),u(t)\rangle \\
		&- \langle C(t)\big(y_\theta(t)-y_d(t)\big),C(t)y_d(t)\rangle,
	\end{split}
\end{equation*}
which implies that
\begin{equation*}
	\begin{split}
		\frac{d}{dt}&\langle P_\theta(t)y_\theta(t)+r_\theta(t),y(t)\rangle
		\geqslant 
		 -\ell(t,y(t),u(t)) - \ell(t,y_\theta(t),u_\theta(t))- \langle C(t)\big(y_\theta(t)-y_d(t)\big),C(t)y_d(t)\rangle\\
		&= -\omega(t,y(t),u(t)) - 2\ell(t,y_\theta(t),u_\theta(t))- \langle C(t)\big(y_\theta(t)-y_d(t)\big),C(t)y_d(t)\rangle.
	\end{split}
\end{equation*}
Combined with \eqref{1}, this yields
\begin{equation*}
	\frac{d}{dt}\langle P_\theta(t)y_\theta(t)+r_\theta(t),y(t)-y_\theta(t)\rangle+\omega(t,y(t),u(t))\geqslant 0.
\end{equation*}
Integrating the above inequality from $ t_0 $ to $ t_1 $, we infer that 
\begin{equation*}
	\begin{split}
		\langle -P_\theta(t_0)y_\theta(t_0)-r_\theta(t_0),y(t_0)-y_\theta(t_0)\rangle+\int_{t_0}^{t_1}\omega(t,y(t),u(t))\,dt 
		\geqslant
		\langle -P_\theta(t_1)y_\theta(t_1)-r_\theta(t_1),y(t_1)-y_\theta(t_1)\rangle,
	\end{split}
\end{equation*}
which gives the dissipativity property \eqref{dissinequa}. 
\end{remark}

\bigskip

We finally show that the evolution operator $U_T (\cdot, \cdot)$ generated by $A(\cdot) - B(\cdot)Q^{-1}(\cdot)B^{*}(\cdot) P^T (\cdot)$ is exponentially stable. 
\begin{proposition}\label{prop3.3}
Under the assumptions of  Lemma \ref{ric1}, there exist two positive constants $M, \omega$ such that
\begin{equation*}
	\|U_T(t,s)\| \leqslant Me^{-\omega (t-s)}, \; \; \forall \; 0\leqslant s \leqslant t \leqslant T.
\end{equation*}
\end{proposition}

\begin{proof}
The proof borrows arguments from \cite[Section 2]{PZ} (see also \cite[Lemma 18]{GL}).
Since the operator $A(\cdot)-B(\cdot) Q^{-1}(\cdot)B^*(\cdot)P_\theta(\cdot)$ is exponentially stable, there exist $M>0$ and  $\rho>0$ such that
\begin{equation*}
	\left\|U_\theta(t, s)\right\| \leqslant M e^{-\rho(t-s)},\;\;0\leqslant s \leqslant t \leqslant T.
\end{equation*}
For each $h \in H$, let $x(\cdot) \in C([\tau,T];H)$ be the solution of
\begin{equation*}
	\left\{\begin{array}{l}
		\dot{x}(t)=\left(A(t)-B(t) Q^{-1}(t) B^*(t) P^T(t)\right) x(t) ,\\
		x(\tau)=h.
	\end{array}\right.
\end{equation*}
Fix a constant $\lambda \in (0, \rho)$ so that $A(\cdot)-B(\cdot) Q^{-1}(\cdot)B^*(\cdot)P_\theta(\cdot)+\lambda I$ is also exponentially stable. Let $y(t)=e^{\lambda (t-\tau)} x(t)$, $\tau\leqslant t\leqslant T$. A straightforward computation shows that
\begin{equation*}
	\left\{\begin{array}{l}
		\dot{y}(t)=\left(A(t)-B(t) Q^{-1}(t) B^*(t) P_\theta(t)+\lambda I\right) y(t)+B(t) Q^{-1}(t) B^*(t)\left[P_\theta(t)-P^T(t)\right] y(t) ,\\
		y(\tau)=h.
	\end{array}\right.
\end{equation*}
The solution of the above equation is given by
\begin{equation*}
	y(t)=U_{F_\theta+\lambda I}(t, \tau) h+\int_\tau^t U_{F_\theta+\lambda I}(t, s) B(s) Q^{-1}(s) B^*(s)\left[P_\theta(s)-P^T(s)\right] y(s) ds, \;\; \tau \leq t \leq T.
\end{equation*}
We obtain from $\eqref{prop1}$ and $\eqref{prop2}$ that
\begin{equation}\label{*}
	\begin{aligned}
		\|y(t)\| &\leqslant M e^{-(\rho-\lambda)(t-\tau)}\|h\|+\int_\tau^t M e^{-(\rho-\lambda)(t-s)}\left\|B(s) Q^{-1}(s) B^*(s)\right\| \left\|P_\theta(s)-P^T(s)\right\| \|y(s)\| ds \\ 
		&\leqslant M_1\|h\|+M_2 \max_{0\leq t\leq \theta}\left\|B(t) Q^{-1}(t) B^*(t)\right\|  \max_{0\leq t\leq \theta}\left\|P_\theta(t)\right\|\int_\tau^t e^{-(\rho-\lambda)(t-s)} e^{-\mu(T-s)}\|y(s)\|ds\\  
		&\leqslant C_1 \|h\|+C_2 e^{-\mu(T-t)} \int_\tau^t \|y(s)\| ds.
	\end{aligned}
\end{equation}

Now, we fix a constant $S>0$ so that $C_2 e^{-\mu S} < \lambda$. To end the proof, we distinguish between three cases.  

\textbf{Case 1.} When $ T-S \leqslant \tau \leqslant t\leqslant T.$

We obtain from $\eqref{*}$ that
\begin{equation*}
	\|y(t)\| \leq C_1\|h\|+C_2 \int_\tau^t\|y(s)\| d s.
\end{equation*}
Since $t-\tau \leqslant S$, by applying the Gronwall inequality, we get
\begin{equation*}
	\|y(t)\| \leqslant C_1\|h\| e^{C_2(t-\tau)} \leqslant C_1 e^{C_2 S}\|h\|.
\end{equation*}
We deduce that
\begin{equation}
	\begin{aligned}\label{case1}
		\left\|U_{A-B Q^{-1} B^* P^T}(t, \tau) h\right\|&=\|x(t)\| \leqslant\|y(t)\| \leqslant C_1 e^{C_2 S}\|h\| \\
		& \leq C_1 e^{\left(C_2+1\right) S} e^{-(t-\tau)}\|h\|. 
	\end{aligned}
\end{equation}

\textbf{Case 2.} When $\tau \leqslant t \leqslant T-S$.

From $\eqref{*}$, we obtain
\begin{equation*}
	\|y(t)\| \leqslant C_1\|h\|+C_2 e^{-\mu S} \int_\tau^t\|y(s)\| d s.
\end{equation*}
By applying the Gronwall inequality, we get
\begin{equation*}
	\|y(t)\| \leqslant C_1 e^{C_2 e^{-\mu S}(t-\tau)}\|h\|.
\end{equation*}
Recalling that $C_2 e^{-\mu S}<\lambda$, we obtain
\begin{equation}
	\begin{aligned}\label{case2}
		\left\|U_{A-B Q^{-1} B^* P^T}(t, \tau)h\right\|&=\|x(t)\|=e^{-\lambda(t-\tau)}\|y(t)\| \\ 
		&\leqslant C_1 e^{-\left(\lambda-C_2 e^{-\mu S} \right) (t-\tau)}\|h\|.
	\end{aligned}
\end{equation}

\textbf{Case 3.} When $\tau < T-S < t \leqslant T$. 

We infer from the definition of the evolution operator that
\begin{equation*}
	U_{A-B Q^{-1} B^* P^T}(t, \tau)=U_{A-B Q^{-1} B^* P^T}(t, T-S)U_{A-B Q^{-1} B^* P^T}(T-S, \tau).
\end{equation*}
Applying Case 1 to  $U_{A-B Q^{-1} B^* P^T}(T-S, \tau)$ and Case 2 to $U_{A-B Q^{-1} B^* P^T}(t, T-S)$, we get
\begin{equation*}
	\begin{aligned}
		\left\|U_{A-B Q^{-1} B^* P^T}(t, \tau)h\right\|&=\left\|U_{A-B Q^{-1} B^* P^T}(t, T-S)U_{A-B Q^{-1} B^* P^T}(T-S, \tau)h\right\|\\
		&\leqslant \left\|U_{A-B Q^{-1} B^* P^T}(t, T-S)\right\| \left\|U_{A-B Q^{-1} B^* P^T}(T-S, \tau)h\right\|\\
		&\leqslant C_1 e^{(C_2+1)S} C_1 e^{\left(C_2 e^{-\mu S}-\lambda \right) (T-S-\tau)}\|h\|\\
		&= C_1^2 e^{(C_2+1)S} e^{-\left(C_2 e^{-\mu S}-\lambda \right)S} e^{\left(C_2 e^{-\mu S}-\lambda \right) (T-\tau)} \|h\|\\
		&\leqslant C_1^2 e^{(C_2+1)S} e^{-\left(C_2 e^{-\mu S}-\lambda \right)S} e^{-\left(\lambda-C_2 e^{-\mu S} \right) (t-\tau)} \|h\|.
	\end{aligned}
\end{equation*}
This estimate, along with $\eqref{case1}$ and $\eqref{case2}$, leads to
\begin{equation*}
	\|U_T(t,s)\| \leqslant Me^{-\omega (t-s)}, \; \; \forall \; 0\leqslant s \leqslant t \leqslant T,
\end{equation*}
for some suitable positive constants $M$ and $\omega$ not depending on $ T $. It completes the proof.
\end{proof}

\section{Proof of Theorem~\ref{periodictarget}}\label{proof}
In order to prove the exponential  periodic turnpike property, we represent $y_\theta(\cdot)$ and $y^T(\cdot)$ in terms of the evolution operator  $U_\theta(\cdot,\cdot)$. For each $t \in (0,T)$, according to Lemma~$ \ref{du120505} $, we can write
\begin{equation}\label{y_theta}
\begin{aligned}
	\;\;\;\;y_\theta(t) 
	= -U_\theta(t,0)\int_{-\infty}^0 U_\theta(0,s)B(s)Q^{-1}(s)B^*(s)r_\theta(s)  ds-\int_0^t U_\theta(t,s)B(s)Q^{-1}(s)B^*(s)r_\theta(s)ds,
\end{aligned}
\end{equation}
and for each $t \in (0,T)$, we write $ \eqref{y_T_equ} $ as 
\begin{equation*}
\dot{y}^T(t) = \big(A(t)-B(t)Q^{-1}(t)B^*(t)P_\theta       (t)\big)y^T(t)+B(t)Q^{-1}(t)B^*(t)\Big[\left(P_\theta(s)-P^T(s)\right)y^T(s)-r^T(t)\Big],
\end{equation*}
hence
\begin{equation}\label{y_T}
y^T(t) = U_\theta(t,0)y_0 + \int_0^t U_\theta(t,s)B(s)Q^{-1}(s)B^*(s)\Big[ \left(P_\theta(s)-P^T(s)\right)y^T(s)-r^T(s)\Big] ds.
\end{equation}

From $\eqref{y_theta}$ and $\eqref{y_T}$, for each $t \in (0,T)$, we infer that
\begin{equation*}
	\begin{aligned}
		y_\theta(t) - y^T(t) &= U_\theta(t,0)\left(y_\theta(0)-y_0\right)-\\&\quad\int_0^t U_\theta(t,s)B(s)Q^{-1}(s)B^*(s)\left[\left( r_\theta(s)-r^T(s)\right) + \left(P_\theta(s)-P^T(s)\right)y^T(s)\right] ds\\
		& \triangleq I_1+I_2.
	\end{aligned}
\end{equation*}
Notice that $z(\cdot)=\left( r_\theta(\cdot)-r^T(\cdot)\right)+\big(P_\theta(\cdot)-P^T(\cdot)\big)y^T(\cdot)$ is the solution of $\eqref{adj}$ with $z(T_0)=r_\theta(T-[T/\theta]\theta)+P_\theta(T-[T/\theta]\theta)y^T(T)$, we have the estimates 
\begin{equation*}
	\|I_1\| \leqslant M e^{- \rho t}\|y_\theta(0)-h\| \triangleq C_1\|y_\theta(0)-y_0\|e^{- \rho t}
\end{equation*}
and
\begin{equation*}
	\|I_2\| \leqslant \int_0^t M e^{-\rho(t-s)} \|B(s)Q^{-1}(s)B^*(s)\|\|z(s)\|ds.
\end{equation*}
We infer from $\eqref{adjexp}$ that
\begin{equation*}
	\begin{aligned}
		\|I_2\| &\leqslant \int_0^t M e^{-\rho(t-s)} \|B(s)Q^{-1}(s)B^*(s)\| M e^{-\rho(T-s)} \|r_\theta(T-[T/\theta]\theta)+P_\theta(T-[T/\theta]\theta)y^T(T)\|ds\\
		&\leqslant M^2 \max_{0\leq t \leq \theta}\|B(t)Q^{-1}(t)B^*(t)\| \max_{0\leq t \leq \theta}\big(\|P_\theta(t)\|\|y^T(T)\|+\|r_\theta(t)\|\big) \int_0^t  e^{-\rho(t-s)} e^{-\rho(T-s)} ds\\
		&\leqslant C\left( 1+\|y^T(T)\|\right) e^{-\rho(T-t)}\\
		& \leqslant C_2\left( 1+\|y_0\|\right) e^{-\rho(T-t)} ,
	\end{aligned}
\end{equation*}
where the last inequality is obtained thanks to Proposition~$\ref{prop3.3}$.
Hence, the above two estimates imply that
\begin{equation}\label{turnpike_y}
	\left\|y_\theta(t)-y^T(t)\right\| \leqslant 
	C_1\|y_\theta(0)-y_0\|e^{- \rho t} + C_2\left( 1+\|y_0\|\right) e^{-\rho(T-t)}
	, \;\; \forall t \in(0,T),
\end{equation}
where $C_1$ and $C_2$ are positive constants not depending on $T$. 

Next, we obtain from $\eqref{adT}$ and $\eqref{adp}$ that
\begin{equation}\label{turnpike_p}
	\begin{aligned}
		\left\| \lambda^T(t)-\lambda_\theta(t)\right\| &=\left\|P_\theta(t)\left(y^T(t)-y_\theta(t) \right) -z(t) \right\| \\ 
		&\leqslant \max_{0 \leq t \leq \theta}\|P_\theta(t)\|\left\|y_\theta(t)-y^T(t)\right\|+\left\|z(t)\right\|\\ 
		&\leqslant  
		C_3\|y_\theta(0)-y_0\|e^{- \rho t} + C_4\left( 1+\|y_0\|\right)e^{-\rho(T-t)}, \;\; \forall t \in(0,T),
	\end{aligned}
\end{equation}
where $C_3$ and $C_4$ are positive constants not depending on $T$.

Finally, we obtain from $\eqref{Toc}$ and $\eqref{poc}$ that
\begin{equation}\label{turnpike_u}
	\begin{aligned} 
		\|u_\theta(t)&-u^T(t)\| \leqslant
		\left\|Q^{-1}(t)B^*(t)\right\|
		\left\|P_\theta(t)\left(y^T(t)-y_\theta(t) \right) -z(t) \right\|\\
		&\leqslant  
		\left\|Q^{-1}(t)B^*(t)\right\|\big[\left\|P_\theta(t)\big(y_\theta(t)-y^T(t)\big)\right\|+\left\|z(t)\right\|\big]\\
		&\leqslant  
		\max_{0 \leq t \leq \theta}\left\|Q^{-1}(t)B^*(t)\right\|\Big[\max_{0 \leq t \leq \theta}\|P_\theta(t)\|\left\|y_\theta(t)-y^T(t)\right\|+\left\|z(t)\right\|\Big]\\
		&\leqslant  
		C\left(\left\|y_\theta(t)-y^T(t)\right\|+\left\|z(t)\right\|\right)\\ 
		&\leqslant  
		C_5\|y_\theta(0)-y_0\|e^{- \rho t} + C_6\left( 1+\|y_0\|\right)e^{-\rho(T-t)}, \;\; \forall t \in(0,T),
	\end{aligned}
\end{equation}
where $C_5$ and $C_6$ are positive constants not depending on $T$.
This estimate, combined with $\eqref{turnpike_y}$ and $\eqref{turnpike_p}$, finally leads to the exponential periodic turnpike inequality \eqref{10271}. 

\section{Further comments}\label{comment}
\subsection{Nonlinear case}\label{nonlinear}
In this paper, we established the globally exponential and periodic turnpike property for linear quadratic periodic optimal control problems in an abstract framework.  The locally periodic turnpike property for 
finite-dimensional nonlinear cases could be obtained 
as in \cite{PZ2,TZ1,TZZ} by linearization along the optimal periodic trajectory, under some exponential stabilizability and detectability assumptions, as well as some smallness assumptions.  As for nonlinear infinite-dimensional case, however, due to the lack of compactness, it seems difficult to establish the periodic turnpike result, and the question is open.

\subsection{Unbounded control operators}
An open and challenging problem is to extend our results to unbounded control operators. This situation involves the theory of differential Riccati equations with unbounded control operators which is incomplete so far. We refer the reader to the case of analytic semigroups in \cite{TZZ}, however, we have no idea if such an extension is feasible in the  periodic case.



\begin{thebibliography}{00}

\bibitem{AT}
P. Acquistapace, Some existence and regularity results for abstract nonautonomous parabolic equations, J. Math. Anal. Appl., 99 (1984), pp. 9–64.

\bibitem{AL}
Z. Artstein and A. Leizarowitz, Tracking periodic signals with the overtaking criterion, IEEE Trans. Automat. Control, 30 (1985), pp. 1123–1126.


\bibitem{BP}
V. Barbu and N. H. Pavel, Periodic optimal control in {H}ilbert space, Appl. Math. Optim., 33 (1996), pp. 169–188.

\bibitem{BPDM}
A. Bensoussan, G. Da Prato, M. C. Delfour, and S. K. Mitter, Representation and Control of Infinite Dimensional Systems, Systems \& Control: Foundations \& Applications, Birkhäuser Boston, Inc., Boston, MA, second ed., 2007.

\bibitem{CP}
R. F. Curtain and A. J. Pritchard, Infinite Dimensional Linear Systems Theory, Vol. 8 of Lecture Notes in Control and Information Sciences, Springer-Verlag, Berlin-New York, 1978.

\bibitem{D}
G. Da Prato, Synthesis of optimal control for an infinite-dimensional periodic problem, SIAM J. Control Optim., 25 (1987), pp. 706–714.

\bibitem{PI}
G. Da Prato and A. Ichikawa, Quadratic control for linear periodic systems, Appl. Math. Optim., 18 (1988), pp. 39–66.

\bibitem{PI1}
G. Da Prato and A. Ichikawa, Quadratic control for linear time-varying systems, SIAM J. Control Optim., 28 (1990), pp. 359-381.

\bibitem{DGSW}
T. Damm, L. Gr\"une, M. Stieler, and K. Worthmann, An exponential turnpike theorem for dissipative discrete time optimal control problems, SIAM J. Control Optim., 52 (2014), 1935–1957.

\bibitem{DR}
R. Datko, Uniform asymptotic stability of evolutionary processes in a Banach space, SIAM J. Math. Anal., 3 (1972), pp. 428–445.

\bibitem{FFOW}
T. Faulwasser, K. Fla{\ss}kamp, S. Ober-Bl\"obaum, K. Worthmann, Towards velocity turnpikes in optimal control of mechanical systems, IFAC-PapersOnLine, 52(2019), pp. 490–495. In Proc. 11th IFAC Symposium on Nonlinear Control Systems, NOLCOS 2019.

\bibitem{GG1}
L. Gr\"une and R. Guglielmi, Turnpike properties and strict dissipativity for discrete time linear quadratic optimal control problems, SIAM J. Control Optim., 56 (2018), pp. 1282-1302.

\bibitem{GG2}
L. Gr\"une and R. Guglielmi, On the relation between turnpike properties and dissipativity for continuous time linear quadratic optimal control problems, Math. Control Relat. Fields, 11 (2021), pp. 169-188.

\bibitem{GTZ}
M. Gugat, E. Tr\'elat, and E. Zuazua, Optimal Neumann control for the 1D wave equation: finite horizon, infinite horizon, boundary tracking terms and the turnpike property, Systems Control Lett., 90 (2016), pp. 61-70.

\bibitem{GL}
R. Guglielmi and Z. Li, Turnpike property for infinite-dimensional generalized LQ problem, arXiv: 2208.00307, preprint.

\bibitem{AI}
A. Ichikawa, Tracking and regulation of periodic systems, IFAC Proceedings Volumes, 22 (1989), pp. 283–288. 5th IFAC Symposium on Control of Distributed Parameter Systems 1989, Perpignan, France, 26-29 June 1989.

\bibitem{IK}
K. Ito and F. Kappel, Evolution Equations and Approximations, Series on Advances in Mathematics for Applied Sciences, 61. World Scientific Publishing Co., Inc., River Edge, NJ, 2002. 

\bibitem{LY}
X. J. Li and J. M. Yong, Optimal Control Theory for Infinite-Dimensional Systems, Systems \& Control: Foundations \& Applications, Birkh\"auser Boston, Inc., Boston, MA, 1995.

\bibitem{L}
A. Lunardi, Differentiability with respect to $(t, s)$ of the parabolic evolution operator, Israel J. Math., 68 (1989), pp. 161–184.

\bibitem{LM}
L. W. McKenzie, Turnpike theorems for a generalized Leontief model, Econometrica, 31 (1963), pp. 165–180.

\bibitem{AP}
A. Pazy, Semigroups of Linear Operators and Applications to Partial Differential Equations, Vol. 44 of Applied Mathematical Sciences, Springer-Verlag, New York, 1983.

\bibitem{PZ}
A. Porretta and E. Zuazua, Long time versus steady state optimal control, SIAM J. Control Optim., 51 (2013), pp. 4242–4273.

\bibitem{PZ2}
A. Porretta and E. Zuazua,
Remarks on long time versus steady state optimal control, in Mathematical Paradigms of Climate Science. Springer, New York, 2016, pp. 67-89.

\bibitem{PAS}
P. A. Samuelson, The periodic turnpike theorem, Nonlinear Anal., 1 (1976), pp. 3–13.


\bibitem{SWY}
J. Sun, H. Wang, and J. Yong, Turnpike properties for stochastic linear-quadratic optimal control problems, Chin. Ann. Math. Ser. B, 43 (2022), pp. 999–1022.

\bibitem{SY}
J. Sun and J. Yong, Turnpike properties for stochastic linear-quadratic optimal control problems with periodic coefficients, preprint.

\bibitem{HT}
H. Tanabe, Equations of Evolution, Vol. 6 of Monographs and Studies in Mathematics, Pitman (Advanced Publishing Program), Boston, Mass.-London, 1979. Translated from the 
Japanese by N. Mugibayashi and H. Haneda.

\bibitem{TE}
E. Tr\'elat, Contr\^ole Optimal: Th\'eorie \& Applications, Vuibert, Collection Math\'ematiques Concr\`etes, 2005.

\bibitem{T1}
E. Tr\'elat, Linear turnpike theorem, Math. Control Signals Systems, 35 (2023), no. 3, pp. 685–739.


\bibitem{TZ1}
E. Tr\'elat and C. Zhang, Integral and measure-turnpike property for infinite dimensional optimal control problems, Math. Control Signals Systems, 30 (2018), no. 1, Art 3, 34pp.

\bibitem{TZZ}
E. Tr\'elat, C. Zhang, and E. Zuazua, Steady-state and periodic exponential turnpike property for optimal control problems in Hilbert spaces, SIAM J. Control Optim., 56 (2018), pp. 1222–1252.

\bibitem{TZ}
E. Tr\'elat and E. Zuazua, The turnpike property in finite-dimensional nonlinear optimal control, J. Differential Equations, 258 (2015), pp. 81–114.

\bibitem{WX}
G. Wang and Y. Xu, Periodic Feedback Stabilization for Linear Periodic Evolution Equations, Springer Briefs in Mathematics, Springer, Cham; BCAM Basque Center for Applied Mathematics, Bilbao, 2016.

\bibitem{JCW}
J. C. Willems, Dissipative dynamical systems. Part I: {G}eneral theory, Arch. Rational Mech. Anal., 45 (1972), 321-351.

\bibitem{Xu}
Y. Xu, Characterization by detectability inequality for periodic stabilization of linear time-periodic evolution systems, Systems Control Lett., 149 (2021), no. 104871, 7pp. 

\bibitem{ZMG}
M. Zanon, L. Gr\"{u}ne, and M. Diehl, Periodic optimal control, dissipativity and MPC, IEEE Trans. Automat. Control, 62 (2017), pp. 2943–2949.

\bibitem{za}
A. J. Zaslavski, Turnpike Theory of Continuous Time Linear Optimal Control Problems, Springer Optimization and Its Applications, 104. Springer International Publishing, Cham, 1st ed., 2015.

\end{thebibliography}
\end{document}